\documentclass[12pt]{amsart}
\newtheorem{dfi}{Definition}[section]

\newtheorem{pro}[dfi]{Proposition}
\newtheorem{thm}[dfi]{Theorem}
\newtheorem{cor}[dfi]{Corollary}
\newtheorem{rem}[dfi]{Remark}
\newtheorem{ex}[dfi]{Example}
\newcommand{\inc}{\hookrightarrow}

\begin{document}
\setlength{\baselineskip}{16pt}

\title[Nonlinear subsets of function spaces and spaceability]{Nonlinear subsets of function spaces and spaceability}
\author[C\'esar Ruiz and V\'\i ctor M. S\'anchez]{C\'esar Ruiz$^*$ and V\'\i ctor M. S\'anchez$^*$}
\address[C\'esar Ruiz]{Department of Mathematical Analysis, Faculty of Mathematics, Complutense University of Madrid, 28040 Madrid, Spain.}
\email{Cesar\_Ruiz@mat.ucm.es}
\address[V\'\i ctor M. S\'anchez]{Institute for Interdisciplinary Mathematics, Department of Mathematical Analysis, Faculty of Mathematics, Complutense University of Madrid, 28040 Madrid, Spain.}
\email{victorms@mat.ucm.es}
\thanks{$^*$Partially supported by the Spanish Ministry of Economy and Competitiveness, grant MTM2012-31286.}
\begin{abstract}
In this paper, we study the existence of infinite dimensional closed linear subspaces of a rearrangement invariant space on $[0,1]$ every nonzero element of which does not belong to any included rearrangement invariant space of the same class such that the inclusion operator is disjointly strictly singular. We consider Lorentz, Marcinkiewicz and Orlicz spaces. The answer is affirmative for Marcinkiewicz spaces and negative for Lorentz and Orlicz spaces. Also, the same problem is studied for Nakano spaces assuming different hypothesis.
\end{abstract}
\maketitle

\section{Introduction}

In the last years the study of large algebraic structures (vector spaces, among others) of functions enjoying certain special properties has become a trend in mathematical analysis. Given a topological vector space $X$, we say that a subset $M\subset X$ is {\em spaceable} if $M\cup\{0\}$ contains an infinite-dimensional closed linear subspace. The subset $M$ will be called {\em lineable} if $M\cup\{0\}$ contains an infinite-dimensional linear subspace (not necessarily closed).

The terms ``lineability" and ``spaceability" were originally coined by V.I. Gurariy and they first appeared in \cite{AGS}. After this first appearance of the notion, many authors became interested in it. See, for instance, the recent works \cite{AGPS,APS,AS,BG,CS,EGS,GQ}, just to cite some. Prior to the publication of \cite{AGS} some authors (when working with infinite-dimensional spaces) already found large algebraic structures enjoying special properties (even though they did not explicitly used terms like lineability or spaceability). Probably the very first result illustrating this was due to B. Levine and D. Milman (\cite{LM}). They proved that the subset of $C[0,1]$ of all functions of bounded variation is not spaceable. Later, V.I. Gurariy proved in \cite{G1} that the set of everywhere differentiable functions on $[0,1]$ is not spaceable, and that there exist infinite-dimensional closed subspaces of $C[0,1]$ all of whose members are differentiable on $(0,1)$.

Within the context of subsets of continuous functions, V.I. Gurariy showed in \cite{G2} that the set of continuous nowhere differentiable functions on $[0,1]$ is lineable. Soon after, V.P. Fonf, V.I. Gurariy and M.I. Kadets showed in \cite{FGK} that the set of continuous nowhere differentiable functions on $[0,1]$ is spaceable in $C[0,1]$. Actually, much more is known about this set. L. Rodr\'\i guez-Piazza showed in \cite{R} that the space constructed in \cite{FGK} can be chosen to be isometrically isomorphic to any separable Banach space. More recently, S. Hencl showed in \cite{Hen} that any separable Banach space is isometrically isomorphic to a subspace of $C[0,1]$ whose non-zero elements are nowhere approximately differentiable and nowhere H\"older. Another set that has also attracted the attention of some authors is the set of differentiable nowhere monotone functions on $\mathbb R$, which was proved to be lineable (see, e.g., \cite{AGS,GMSS}).

We refer the interested reader to a recent survey on the topic (\cite{BPS}), where many more examples can be found and techniques are developed in several different frameworks.

The study of subspaces of $L^p$-spaces is a classical topic in Banach and quasi-Banach space theory, which is related to spaceability (see, for instance, \cite{BO,BCFPS}). Recently, it was proved in \cite{BFDS} that the set $L^p[0,1]\setminus\bigcup\limits_{q>p}L^q[0,1]$ is spaceable for every $p>0$. This problem can be considered for other function spaces. In this paper the problem is studied and solved for Lorentz, Marcinkiewicz and Orlicz spaces on $[0,1]$, which are rearrangement invariant spaces. Since the continuous inclusion operator $L^q\inc L^p$ is disjointly strictly singular for $1\le p<q<\infty$, that is, there is no disjoint sequence of non-null functions such that the closed subspace spanned by them in both spaces are isomorphic, the natural question involves studying if it is spaceable the set formed by the functions belonging to a Lorentz (Marcinkiewicz, Orlicz) space and not belonging to any included Lorentz (Marcinkiewicz, Orlicz) space such that the inclusion operator is disjointly strictly singular. We have spaceability for Marcinkiewicz spaces and not for Lorentz and Orlicz spaces. We obtain these results proving previously a general one on spaceability in rearrangement invariant spaces (Theorem \ref1). For Orlicz spaces we also obtain affirmative results considering only a sequence of spaces, and with enough separation between the functions that generate the spaces through using indices of such functions.

On the other hand, and as a kind of generalization of the original result in \cite{BFDS}, we study if it is spaceable the set formed by the functions belonging to a Nakano function space and not belonging to any included Nakano function space, assuming different hypothesis because a characterization of disjointly strictly singular inclusion operators between Nakano function spaces is still unknown.

\section{Preliminaries}

Let us give some definitions and notations. We consider the interval $[0,1]$ and the Lebesgue measure $\lambda$. The {\em distribution function} $\lambda_x$ associated to a measurable function $x$ on $[0,1]$ is defined by
	$$\lambda_x(s)=\lambda\{t\in[0,1]:|x(t)|>s\}.$$
Related to it we have the {\em decreasing rearrangement function} $x^*$ of $x$ which is defined by
	$$x^*(t)=\inf\{s\in[0,\infty):\lambda_x(s)\le t\}.$$

A Banach space $E$ of measurable functions defined on $[0,1]$ is said to be a {\em rearrangement invariant space} (briefly r.i. space) if the following conditions are satisfied:
\begin{enumerate}
\item If $y\in E$ and $|x(t)|\le|y(t)|$ $\lambda$-a.e. on $[0,1]$, then $x\in E$ and $\|x\|_E\le\|y\|_E$.
\item If $y\in E$ and $\lambda_x=\lambda_y$, then $x\in E$ and $\|x\|_E=\|y\|_E$.
\end{enumerate}

Important examples of r.i. spaces are Lorentz, Marcinkiewicz and Orlicz spaces with the corresponding norms.

Let us denote by $\Phi$ the class of all increasing concave functions $\phi$ on $[0,1]$ such that $\phi(0)=0$. If $\phi\in\Phi$, the {\em Lorentz space} $\Lambda(\phi)$ consists of all measurable functions $x$ defined on $[0,1]$ such that
	$$\|x\|_{\Lambda(\phi)}=\int_0^1x^*(t)\ d\phi(t)<\infty.$$
And the {\em Marcinkiewicz space} $M(\phi)$ consists of all measurable functions $x$ defined on $[0,1]$ for which
	$$\|x\|_{M(\phi)}=\sup_{0<t\le1}\frac{\int_0^tx^*(s)\ ds}{\phi(t)}<\infty.$$

The {\em fundamental function} $\phi_E$ of an r.i. space $E$ is defined by $\phi_E(t)=\left\|\chi_{[0,t]}\right\|_E$ with $0\le t\le1$. Given $\phi\in\Phi$, the spaces $\Lambda(\phi)$ and $M(\tilde{\phi})$ are respectively the smallest and the biggest r.i. space having the same fundamental function $\phi$, being $\tilde{\phi}(t)=\frac t{\phi(t)}$, for $0<t\le1$ and $\tilde{\phi}(0)=0$.

Now, let us denote by $\Psi$ the class of all positive convex functions $\psi$ on $[0,\infty)$ such that $\psi(0)=0$ and $\lim\limits_{t\to\infty}\psi(t)=\infty$. If $\psi\in\Psi$, the {\em Orlicz space} $L^\psi$ consists of all measurable functions $x$ defined on $[0,1]$ for which
	$$\|x\|_{L^\psi}=\inf\left\{s>0:\int_0^1\psi\left(\frac{|x(t)|}s\right)\ dt\le1\right\}<\infty.$$

We will say that $\psi\in\Psi$ verifies the {\em$\Delta_2$-condition at $\infty$} if
	$$\limsup_{t\to\infty}\frac{\psi(2t)}{\psi(t)}<\infty.$$

And we define the indices $1\le p_\psi^\infty\le q_\psi^\infty\le\infty$ by
	$$p_\psi^\infty=\sup\left\{p>0:\frac{\psi(t)}{t^p}\mbox{ is increasing for large }t\right\}$$
and
	$$q_\psi^\infty=\inf\left\{q>0:\frac{\psi(t)}{t^q}\mbox{ is decreasing for large }t\right\}.$$

For properties of r.i. spaces we refer to \cite{BS,KPS,LT}.

Now, given a measurable function $p:[0,1]\longrightarrow[1,\infty)$, we write
	$$p^-=\mbox{ess inf}\{p(t):t\in[0,1]\}$$
and
	$$p^+=\mbox{ess sup}\{p(t):t\in[0,1]\}.$$

We consider {\em Nakano function spaces} (or {\em variable exponent Lebesgue spaces}) $L^{p(\cdot)}$, the set of all measurable functions $x$ defined on $[0,1]$ such that
	$$\rho_{p(\cdot)}(x)=\int_0^1|x(t)|^{p(t)}\ dt<\infty.$$
And
	$$\|x\|_{p(\cdot)}=\inf\left\{s>0:\rho_{p(\cdot)}\left(\frac xs\right)\le1\right\}$$
is the Luxemburg norm associated to $L^{p(\cdot)}$.

With the usual pointwise order, $\left(L^{p(\cdot)},\|\cdot\|_{p(\cdot)}\right)$ is a Banach lattice. We will assume that $p^+<\infty$. In this case it is a separable order continuous Banach lattice.

The inclusion $L^{q(\cdot)}\inc L^{p(\cdot)}$ holds if and only if $p(t)\le q(t)$ a.e.

The {\em essential range} of the exponent function $p$ is defined by
	$$R_{p(\cdot)}=\{q\in[1,\infty):\lambda(p^{-1}(q-\epsilon,q+\epsilon))>0\ \forall\epsilon>0\}.$$
The essential range is a lattice-isomorphic invariant of Nakano function spaces (see \cite{HRu1}).

Given two sequences $(p_n)_{n\in\mathbb N}\subset[1,\infty)$ and $(w_n)_{n\in\mathbb N}\subset(0,\infty)$, the {\em weighted Nakano sequence
space} is defined by
	$$\ell_{(p_n)}(w_n)=\left\{x=(x_n)_{n\in\mathbb N}\subset\mathbb R:\sum_{n=1}^\infty\left|\frac{x_n}s\right|^{p_n}w_n<\infty\mbox{ for some }s>0\right\}$$
equipped with the Luxemburg norm
	$$\|x\|_{\ell_{(p_n)}(w_n)}=\inf\left\{s>0:\sum_{n=1}^\infty\left|\frac{x_n}s\right|^{p_n}w_n\le1\right\}.$$
In particular, if $p_n=p$ for every $n\in\mathbb N$, we have the spaces $\ell_p(w_n)$. And when $w_n=1$ for every $n\in\mathbb N$, we write $\ell_{(p_n)}$, the classical Nakano sequence space.

For properties of Nakano spaces we refer to \cite{DHHR,HRu1,M,Na,Ne}.

A linear operator between two Banach spaces $X$ and $Y$ is {\em strictly singular} if it fails to be an isomorphism on any infinite dimensional subspace of $X$. A weaker notion for Banach lattices introduced in \cite{HRo1} is the following one: an operator $T$ from a Banach lattice $X$ to a  Banach space $Y$ is said to be {\em disjointly strictly singular} (DSS in short) if there is no disjoint sequence of non-null vectors in $X$ such that the restriction of $T$ to the closed subspace spanned by them is an isomorphism. The notion of DSS has turned out to be a useful tool in the study of lattice structure of function spaces. See, for instance, \cite{Her,HRo1,HRo2}.

If $\phi,\psi\in\Phi$, we will say that $\phi\ll\psi$ if $\lim\limits_{t\to0}\frac{\phi(t)}{\psi(t)}=0$. In the case of Lorentz spaces $\Lambda(\phi)\inc\Lambda(\psi)$ holds if and only if there exists $C\in\mathbb R$ such that $\psi\le C\phi$. And the inclusion operator is DSS if and only if $\psi\ll\phi$. For Marcinkiewicz spaces the results are analogous: $M(\phi)\inc M(\psi)$ holds if and only if there exists $C\in\mathbb R$ such that $\phi\le C\psi$. And the inclusion operator is DSS if and only if $\phi\ll\psi$ (see \cite{A}).

For Orlicz spaces, if $\phi,\psi\in\Psi$, $L^\phi\inc L^\psi$ holds if and only if there exist $c>0$ and $T\ge0$ such that $\psi(t)\le c\,\phi(t)$ for every $t\ge T$. And in the case of separable Orlicz spaces, that is, verifying $\phi$ and $\psi$ the $\Delta_2$-condition at $\infty$, the inclusion operator is DSS if and only for every $C>0$ there exist distinct points $x_1,\ldots,x_n\in[1,\infty)$ and $a_1,\ldots,a_n\in(0,\infty)$ such that
	$$\sum_{k=1}^na_k\psi(tx_k)\le C\sum_{k=1}^na_k\phi(tx_k)$$
for all $t\ge1$ (see \cite{HRo1}). We will say that $\psi\prec\phi$ if the previous conditions are satisfied.

\section{Spaceability in r.i. spaces}

Given $a\in[0,1)$, $r\in(0,1-a]$ and a measurable function $x$ defined on $[0,1]$, we define the linear operator $T_{a,r}(x(t))=x\left(\frac{t-a}r\right)\chi_{(a,a+r]}(t)$, which is bounded from $L^\infty$ to $L^\infty$ and from $L^1$ to $L^1$. Thus, using Calder\'on-Mitjagin interpolation theorem (\cite[Theorem 2.a.10]{LT}), $T_{a,r}$ is bounded from $E$ to $E$ for every r.i. space $E$.

Also, we will say that a subset $A$ of an r.i. space $E$ is a {\em$T$-subset} if the following conditions are satisfied:
\begin{enumerate}
\item If $y\in A$ and $|x(t)|\le|y(t)|$ $\lambda$-a.e. on $[0,1]$, then $x\in A$.
\item If there exist $a\in[0,1)$ and $r\in(0,1-a]$ such that $T_{a,r}(x)\in A$, then $x\in A$.
\end{enumerate}

\begin{thm}
\label1
Let $E$ be an r.i. space and $A\subseteq E$ a $T$-subset. The following statements are equivalent:
\begin{enumerate}
\item$E\setminus A$ is spaceable.
\item$E\setminus A\neq\emptyset$.
\end{enumerate}
\end{thm}

\begin{proof}
We only need to show that $E\setminus A$ is spaceable when it is non-empty. Given $x\in E\setminus A$, we consider the sequence of functions $x_n=T_{a_n,r_n}(x)$ for $a_n=1-\frac1{2^{n-1}}$ and $r_n=\frac1{2^n}$. The functions $x_n$ are linearly independent since they have disjoint supports, and $x_n\in E\setminus A$ for every $n\in\mathbb N$. Thus, since
	$$\left|\sum_{n=1}^\infty\lambda_nx_n\right|=\sum_{n=1}^\infty|\lambda_n|x_n\ge|\lambda_i|x_i$$
for every $i\in\mathbb N$, we obtain that
	$$\overline{[x_n:n\in\mathbb N]}\subset E\setminus A.$$
\end{proof}

The union of a collection of r.i. spaces contained in another r.i. space is a $T$-subset. Thus, using the previous result and the density in a separable r.i. space of the set formed by all simple functions, we have the following consequence:

\begin{cor}
Let $E$ be a separable r.i. space and $(E_i)_{i\in I}$ a collection of r.i. spaces contained in $E$. The following statements are equivalent:
\begin{enumerate}
\item$E\setminus\bigcup\limits_{i\in I}E_i$ is spaceable.
\item$E\setminus\bigcup\limits_{i\in I}E_i\neq\emptyset$.
\item$\bigcup\limits_{i\in I}E_i$ is not closed in $E$.
\end{enumerate}
\end{cor}

The above corollary is related to the general sufficient condition for spaceability obtained by D. Kitson and R.M. Timoney in \cite[Theorem 3.3]{KT}. Next example shows that we need additional hypotheses to obtain an equivalence.

\begin{ex}
Let $1\le p<q<2$. In $L^p$ we can find a closed subspace $A$ isomorphic to $L^q$. If $(x_n)_{n\in\mathbb N}$ is a sequence in $L^p$ equivalent to the canonical basis of $\ell_p$, then $\overline{[x_n:n\in\mathbb N]}\cap A$ is a finite-dimensional closed subspace. Then $L^p\setminus A$ is spaceable.
\end{ex}

\section{Spaceability in Lorentz and Marcinkiewicz spaces}

Using Theorem \ref1, we can state an affirmative result on spaceability in some cases. First, we are going to study the problem for Marcinkiewicz spaces.

\begin{thm}
The set $M(\psi)\setminus\bigcup\limits_{\phi\ll\psi}M(\phi)$ is spaceable for every $\psi\in\Phi$.
\end{thm}

\begin{proof}
If $\psi\in\Phi$, then $\psi'\in M(\psi)$ and $\|\psi'\|_{M(\psi)}=1$. Now, given $\phi\in\Phi$ such that $\phi\ll\psi$, we have
	$$\sup_{0<t\le1}\frac{\int_0^t\psi'(s)\ ds}{\phi(t)}=\infty.$$
Therefore, $\psi'\not\in M(\phi)$. Using Theorem \ref1 we obtain the result.
\end{proof}

However, we do not obtain spaceability for Lorentz spaces.

\begin{thm}
The set $\Lambda(\psi)\setminus\bigcup\limits_{\psi\ll\phi}\Lambda(\phi)$ is empty for every $\psi\in\Phi$.
\end{thm}

\begin{proof}
Let $\psi\in\Phi$ and $x\in\Lambda(\psi)$. There exists $\phi\in\Phi$ such that $\psi\ll\phi$ and $x\in\Lambda(\phi)$. Indeed, since $x^*\psi'\in L^1[0,1]$, there exists a function $y=y^*$ such that $\lim\limits_{t\to0}y(t)=\infty$ and
	$$\int_0^1x^*(s)\psi'(s)y(s)\ ds<\infty.$$
We take on $[0,1]$ the function
	$$\phi(t)=\int_0^t\psi'(s)y(s)\ ds.$$
The function $\phi\in\Phi\cap L^1[0,1]$, $x\in\Lambda(\phi)$ and
\begin{eqnarray*}
\lim_{t\to0}\frac{\psi(t)}{\phi(t)}&=&\lim_{t\to0}\frac{\int_0^t\psi'(s)\ ds}{\int_0^t\psi'(s)y(s)\ ds}\\&\le&\lim_{t\to0}\frac{\int_0^t\psi'(s)\ ds}{y(t)\int_0^t\psi'(s)\ ds}\\&=&\lim_{t\to0}\frac1{y(t)}\\&=&0.
\end{eqnarray*}
\end{proof}

\section{Spaceability in Orlicz spaces}

Neither do we obtain spaceability for Orlicz spaces.

\begin{thm}
The set $L^\psi\setminus\bigcup\limits_{\psi\prec\phi}L^\phi$ is empty for every $\psi\in\Psi$ verifying the $\Delta_2$-condition at $\infty$.
\end{thm}

\begin{proof}
Let $\psi\in\Psi$ verifying the $\Delta_2$-condition at $\infty$ and $x\in L^\psi$. There exists $\phi\in\Psi$ verifying the $\Delta_2$-condition at $\infty$ such that $\psi\prec\phi$ and $x\in L^\phi$. Indeed, we define the disjoint measurable sets
	$$A_n=\{t\in[0,1]:2^n\le|x(t)|<2^{n+1}\}$$
with $n\in\mathbb N$. Then
	$$\sum_{n=1}^{\infty}\psi(2^n)\mu(A_n)\le\sum_{n=1}^{\infty}\int_{A_n}\psi(|x(t)|)\ dt\le\sum_{n=1}^{\infty}\psi(2^{n+1})\mu(A_n).$$
Since $\psi$ verifies the $\Delta_2$-condition at $\infty$, we have $\sum\limits_{n=1}^{\infty}\psi(2^n)\mu(A_n)<\infty$ if and only if $\sum\limits_{n=1}^{\infty}\psi(2^{n+1})\mu(A_n)<\infty$. Now, we choose an increasing sequence of positive numbers $(b_n)_{n\in\mathbb N}$ such that $b_1=1$, $\lim\limits_{n\to\infty}b_n=\infty$, the series $\sum\limits_{n=1}^{\infty}b_n\psi(2^n)\mu(A_n)$ also converges and the sequence $\left(\frac{b_{n+1}}{b_n}\right)_{n\in\mathbb N}$ is bounded. Finally, we define the function
	$$\phi(t)=\left\{\begin{array}{ll}\psi(t)&\mbox{if }0\le t\le2,\\b_n\psi(t)&\mbox{if }2^n<t\le2^{n+1}\mbox{ with }n\in\mathbb N.\end{array}\right.$$
Clearly, the function $\phi$ verifies all the required conditions.
\end{proof}

However, an affirmative result can be obtained considering only a sequence of Orlicz spaces.

\begin{pro}
The set $L^\psi\setminus\bigcup\limits_{n=1}^\infty L^{\phi_n}$ is spaceable for every $\psi\in\Psi$ and $(\phi_n)_{n\in\mathbb N}\subset\Psi$ such that $L^{\phi_n}\inc L^\psi$ and $L^{\phi_n}\ne L^\psi$ for every $n\in\mathbb N$.
\end{pro}

\begin{proof}
Let $\{A_n:n\in\mathbb N\}$ be a partition of $[0,1]$ formed by disjoint intervals. There exists $x_n\in L^\psi(A_n)\setminus L^{\phi_n}(A_n)$ with $\|x_n\|_{L^\psi(A_n)}=1$ for every $n\in\mathbb N$. The function
	$$x=\sum_{n=1}^\infty\frac{x_n}{2^n}\chi_{A_n}$$
belongs to $L^\psi\setminus\bigcup\limits_{n=1}^\infty L^{\phi_n}$. Using Theorem \ref1 we have the result.
\end{proof}

With enough separation between the functions that generate the Orlicz spaces through using indices of such functions, we can obtain another affirmative result.

\begin{pro}
\label2
The set $L^\psi\setminus\bigcup\limits_{q_\psi^\infty<p_\phi^\infty}L^\phi$ is spaceable for every $\psi\in\Psi$ verifying the $\Delta_2$-condition at $\infty$.
\end{pro}

\begin{proof}
Let $\psi\in\Psi$ and a decreasing sequence $q_n\downarrow q_\psi^\infty$. Therefore,
	$$\bigcup_{q_\psi^\infty<p_\phi^\infty}L^\phi=\bigcup_{n=1}^\infty\bigcup_{
q_n<p_\phi^\infty}L^\phi.$$

We define the function
	$$x_n(t)=\frac1{(t-2^{-n-1})^{1/q_n}}\chi_{(2^{-n-1},2^{-n}]}(t)$$
for every $n\in\mathbb N\cup\{0\}$.

There exist $q\in\left[q_\psi^\infty,q_n\right)$ and a constant $C$ such that $\psi(t)\le Ct^q$ for every $t\ge1$. Then
	$$\int_0^1\psi(|x_n(t)|)\ dt\le C\int_{2^{-n-1}}^{2^{-n}}\frac1{(t-2^{-n-1})^{q/q_n}}\ dt<\infty.$$

On the other hand, given $\phi\in\Psi$ such that $q_n<p_\phi^\infty$, there exists a constant $C'$ such that $\phi(t)\ge C't^{q_n}$ for every $t\ge1$. Then, for $a>0$ we have
	$$\int_0^1\phi\left(\frac{|x_n(t)|}a\right)\ dt\ge\frac{C'}{a^{q_n}}\int_{2^{-n-1}}^{t_a}\frac1{t-2^{-n-1}}\ dt=\infty$$
being $t_a$ such that $x_n(t_a)>a$.

Assuming $\|x_n\|_{L^\psi}=1$ for every $n\in\mathbb N$, the function
	$$x=\sum_{n=1}^\infty\frac1{2^n}x_n$$
belongs to $L^\psi\setminus\bigcup\limits_{q_\psi^\infty<p_\phi^\infty}L^\phi$. Using Theorem \ref1 we obtain the result.
\end{proof}

\begin{rem}
{\em If $\psi(t)=t^p$ with $p\ge1$, then $p_\psi^\infty=q_\psi^\infty=p$. Thus, Proposition \ref2 is a kind of generalization for Orlicz spaces of the original result in \cite{BFDS} for $L^p$-spaces.}
\end{rem}

Spaceability in Orlicz spaces has been also studied recently in \cite{AM}.

\section{Spaceability in Nakano spaces}

A characterization of disjointly strictly singular inclusion operators between Nakano function spaces is still unknown. Also, Nakano function spaces are not r.i. spaces. However, using techniques in previous sections, we settle for obtaining a kind of generalization of the original result in \cite{BFDS} for $L^p$-spaces assuming different hypothesis.

By \cite[Proposition 3.9]{HRu2}, the equality $L^{p(\cdot)}=L^{q(\cdot)}$ holds if and only if $p=q$ a.e. Thus,
	$$L^{p(\cdot)}=\bigcup_{q\ge p,\,q\ne p}L^{q(\cdot)}.$$
Indeed, let $f\in L^{p(\cdot)}$. We define the measurable sets $A_n=\{t\in[0,1]:|f(t)|>n\}$ with $n\in\mathbb N$. Since $\lim\limits_{n\to\infty}\lambda(A_n)=0$, there exists $n_0\in\mathbb N$ such that $\lambda(A_{n_0})<1$. If $q(t)=\chi_{A_{n_0}}(t)p(t)+\chi_{[0,1]\setminus A_{n_0}}(t)(p(t)+1)$, then $f\in L^{q(\cdot)}$.

So, we will need to add an extra condition to prove our result for Nakano function spaces.

First, we are going to consider weighted Nakano sequence spaces.

\begin{pro}
\label3
Let two sequences $(p_n)_{n\in\mathbb N}\subset[1,\infty)$ and $(w_n)_{n\in\mathbb N}\subset(0,\infty)$ such that $\lim\limits_{n\to\infty}p_n=p$ and $\sum\limits_{n=1}^\infty w_n<\infty$.
\begin{enumerate}
\item The set $\ell_p(w_n)\setminus\bigcup\limits_{q>p}\ell_q(w_n)$ is spaceable for every $p\ge1$.
\item The set $\ell_{(p_n)}(w_n)\setminus\bigcup\limits_{q_n>p_n,\,\liminf\limits_{n\to\infty}q_n>p}\ell_{(q_n)}(w_n)$ is spaceable for every $p\ge1$.
\end{enumerate}
\end{pro}

\begin{proof}
We will prove only the first part of the proposition. The proof of the second part is analogous.

Let $\{\Omega_{m,n}:m,n\in\mathbb N\}$ be a disjoint partition of $\mathbb N$ formed by sets whose cardinality is the same as $\mathbb N$. We also consider a decreasing sequence $(q_m)_{m\in\mathbb N}$ convergent to $p$.

Let $n\in\mathbb N$. For every $m\in\mathbb N$ there exists $x_{m,n}=\sum\limits_{k\in\Omega_{m,n}}a_ke_k$ such that $\|x_{m,n}\|_{\ell_p(w_n)}=1$ and $\sum\limits_{k\in\Omega_{m,n}}|a_k|^{q_m}w_k=\infty$. Thus,
	$$x_n=\sum_{m=1}^\infty\frac{x_{m,n}}{2^m}\in\ell_p(w_n).$$
However, $x_n\not\in\ell_q(w_n)$ for every $q>p$. Indeed, there exists $m_0\in\mathbb N$ such that $q>q_{m_0}>p$. Then
	$$\sum_{k\in\Omega_{m_0,n}}|a_k|^qw_k=\infty.$$

Reasoning as in the proof of Theorem \ref1 with the disjoint sequence $(x_n)_{n\in\mathbb N}$ we end the proof.
\end{proof}

Using the previous proposition we obtain our result for Nakano function spaces.

\begin{thm}
If $r\in R_{p(\cdot)}$, then the set
	$$L^{p(\cdot)}\setminus\bigcup_{q\ge p,\,r\not\in R_{q(\cdot)}}L^{q(\cdot)}$$
is spaceable.
\end{thm}

\begin{proof}
Let $r\in R_{p(\cdot)}$. By \cite[Proposition 5.1 and Proposition 4.2.a]{HRu2}, there exist a sequence $(p_n)_{n\in\mathbb N}$ and a sequence of disjoint measurable sets $(B_n)_{n\in\mathbb N}$ such that $\lim\limits_{n\to\infty}p_n=r$ and $B_n\subset p^{-1}\left(r-\frac1n,r+\frac1n\right)$. And the function $\sum\limits_{n=1}^\infty a_n\chi_{B_n}$ belongs to $L^{p(\cdot)}$ if and only if $\sum\limits_{n=1}^\infty|a_n|^{p_n}\lambda(B_n)<\infty$.

Now, let $q$ such that $q\ge p$ and $r\not\in R_{q(\cdot)}$. There exists $n_0\in\mathbb N$ verifying $q(t)>r+\frac1{n_0}$ for every $t\in\bigcup\limits_{n>n_0}B_n$. Then
	$$\rho_{q(\cdot)}\left(\sum_{n=1}^\infty a_n\chi_{B_n}\right)\ge\sum_{n>n_0}|a_n|^{r+\frac1{n_0}}\lambda(B_n).$$

Using the second part of Proposition \ref3 we obtain the result.
\end{proof}

\noindent{\bf Open questions.} There are still some unsolved problems related to the content of this paper. The first one could be studying the algebrability of the sets considered. We recall that the set $M$ is said to be {\em algebrable} if $M\cup\{0\}$ contains an infinitely generated algebra.

On the other hand, a characterization of disjointly strictly singular inclusion operators between Nakano function spaces is still unknown. In other case, spaceability in Nakano functions spaces could be studied in the same way as has been done for Lorentz, Marcinkiewicz and Orlicz spaces.

\noindent{\bf Acknowledgement.} The authors wish to thank E.M. Semenov for his valuable suggestions.

\end{document}